\documentclass{amsart}
\usepackage{geometry}
\usepackage{graphicx}
\usepackage{amsrefs}
\usepackage[colorlinks]{hyperref}
\usepackage{mathpazo}

\newtheorem{theorem}{Theorem}[section]

\newtheorem{lemma}[theorem]{Lemma}
\newtheorem{proposition}[theorem]{Proposition}
\theoremstyle{definition}
\newtheorem{definition}[theorem]{Definition}

\newtheorem{example}[theorem]{Example}
\newtheorem{question}[theorem]{Question}
\numberwithin{equation}{section}

\begin{document}
\title{Multivariate writhe polynomial of multi-virtual knots}
\author{Zhiyun Cheng}
\address{School of Mathematical Sciences, Beijing Normal University, Beijing 100875, China}
\email{czy@bnu.edu.cn}
\subjclass[2020]{57K10}
\keywords{multi-virtual knot, multivariate writhe polynomial}
\begin{abstract}
In this paper, we extend the writhe polynomial invariant from virtual knots to multi-virtual knots. Several questions asked in \cite{KMV2026} have been answered.
\end{abstract}
\maketitle

\section{Introduction}\label{section1}
As an extension of classical knot theory, which studies the embeddings of $S^1$ into $S^3$, or equivalently $S^2\times [0, 1]$, the virtual knot theory \cite{Kau1999} studies the embeddings of $S^1$ into $\Sigma_g\times [0, 1]$ up to isotopy, orientation-preserving automorphisms of $\Sigma_g$ and stabilizations. Here $\Sigma_g$ denotes an orientable closed surface of genus $g$. Very recently, Kauffman introduced a generalized version of virtual knot theory, called multi-virtual knot theory \cite{Kau2025}. Roughly speaking, in contrast to virtual knot diagrams, which contain only one type of virtual crossing, a multi-virtual knot diagram admits several distinct types of virtual crossings. Geometrically, a multi-virtual knot can be considered an embedded $S^1$ in $S^2\times [0, 1]$ with some labeled handles attached, up to the merger of handles of the same type and the corresponding stabilizations. The original motivation of introducing the multi-virtual knot theory comes from the investigation of Penrose evaluation for colorings of (not necessary planar) trivalent graphs. The readers are referred to \cite{Kau2025} for more details.

Since Kauffman introduced virtual knot theory at the end of the last century, virtual knot invariants can be roughly divided into two categories: the first kind mainly come from classical knot invariants with some modifications for virtual knots, such as the knot group, knot quandle, generalized Alexander polynomial, arrow polynomial and Miyazawa polynomial; the other kind originates from some intrinsic properties of virtual knots, which usually vanishes on classical knots. The first virtual knot invariant of this kind is the $2\mathbb{Z}$-valued odd writhe introduced by Kauffman in \cite{Kau2004}. Later, it was generalized to the odd writhe polynomial in \cite{Che2014}. A big family of the second kind of virtual knot invariants are the so-called index type invariants. The main idea of such invariants can be traced back to Turaev's work \cite{Tur2004}. One concrete example is the writhe polynomial $W_K(t)$ introduced in \cite{CG2013}, see also \cites{Dye2013,IKL2013,Kau2013,ST2014}. For more recent developments on this topic, we refer the readers to \cite{Che2021} or \cite{CFGMX2020}.

For multi-virtual knots, generalized version of bracket polynomial, Alexander polynomial, quandle and arrow polynomial have been discussed in \cite{Kau2025}. The quandle idea has recently been systematically developed and investigated by Kauffman, Mukherjee and Vojt\v echovsk\'y in \cite{KMV2026}. In particular, they introduced the notion of operator quandles for multi-virtual knots and defined the operator quandle coloring invariants as well as the operator quandle 2-cocycle invariants. The main aim of this paper is to present a generalization of the writhe polynomial for multi-virtual knots, which is a multivariate polynomial with integer coefficients. Several questions posed in \cite{KMV2026} will be answered by using this new invariant.

The rest of this paper is arranged as follows. In Section \ref{section2}, we briefly review several fundamental concepts in virtual knot theory together with the definition of the writhe polynomial. Section \ref{section3} is devoted to introducing the multivariate writhe polynomial for multi-virtual knots and proving that it is indeed a knot invariant. Several concrete computational examples will be presented. Section \ref{section4} contains some further discussions and more examples.  

\section{Virtual knots and writhe polynomial}\label{section2}
\subsection{Virtual knots}
A virtual knot diagram is a classical knot diagram with some classical crossing points replaced by virtual crossing points. Usually, we use a small circle around the crossing to denote a virtual crossing. We say two virtual knot diagrams are equivalent if they can be connected by a sequence of generalized Reidemeister moves, see Figure \ref{figure1}. Virtual knots are defined as the equivalence classes of virtual knot diagrams modulo generalized Reidemeister moves.

\begin{figure}[h]
\centering
\includegraphics{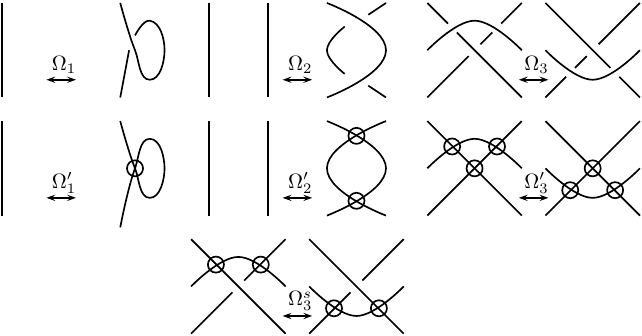}\\
\caption{Generalized Reidemeister moves}\label{figure1}
\end{figure}

Classical knots can be regarded as virtual knots without virtual crossings. It is known that if two classical knot diagrams can be related by finitely many generalized Reidemeister moves, then they also can be related by finitely many classical Reidemeister moves \cite{Kau1999}. In other words, classical knot theory can be embedded in virtual knot theory. From a topological point of view, a virtual knot diagram can be considered as a knot diagram in $S^2$. For each virtual crossing, we may eliminate it by attaching a handle to the 2-sphere, which yields an embedded $S^1$ inside a thickened surface. Now two virtual knots are equivalent if these two embeddings are equivalent modulo isotopy, orientation-preserving automorphisms of the surface and adding or deleting empty handles \cite{CKS2002}. It was proved in \cite{Kup2003} that two virtual knots are equivalent if they are isotopic in the thickened surface with minimal supporting genus. This provides another approach to show that classical knot theory can be regarded as a subset of virtual knot theory.

\subsection{Writhe polynomial}
Roughly speaking, the writhe polynomial of a virtual knot counts the signed sum of indexed classical crossing points. In order to give the definition of writhe polynomial, we need the notion of Gauss diagram. For a given oriented virtual knot diagram $D$, the Gauss diagram $G(D)$ is nothing but the preimage of $D$. It consists of a plane circle oriented in the anticlockwise direction with some oriented and signed chords inside. Here each chord corresponds to a classical crossing point, the orientation is directed from the preimage of the overcrossing to the preimage of the undercrossing, and the sign is exactly that of the crossing. Figure \ref{figure2} depicts a virtual trefoil knot diagram and its corresponding Gauss diagram. Note that all virtual crossings are ignored in the Gauss diagram. One advantage of this is, when we need to check the invariance of something under the generalized Reidemeister moves, we only need to consider the classical Reidemeister moves $\Omega_1, \Omega_2, \Omega_3$, rather than the generalized Reidemeister moves that involve virtual crossings. Although we can not see the virtual crossings in the Gauss diagram, later it will be found that the Gauss diagram indeed contains some information of the virtual crossings.

\begin{figure}[h]
\centering
\includegraphics{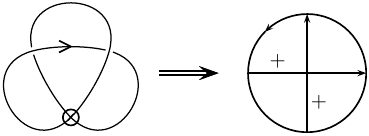}\\
\caption{Virtual trefoil knot and its Gauss diagram}\label{figure2}
\end{figure}

Let $D$ be a virtual knot diagram and $c$ a real crossing point of it. On the premise that no ambiguity will be caused, we will also use $c$ to denote the corresponding chord in the Gauss diagram $G(D)$. For a fixed chord $c$, let us use
\begin{itemize}
  \item $r_+(c)$ to denote the number of positive chords crossing $c$ from left to right;
  \item $r_-(c)$ to denote the number of negative chords crossing $c$ from left to right;
  \item $l_+(c)$ to denote the number of positive chords crossing $c$ from right to left;
  \item $l_-(c)$ to denote the number of negative chords crossing $c$ from right to left.
\end{itemize}
See Figure \ref{figure3}.

\begin{figure}[h]
\centering
\includegraphics[width=3cm]{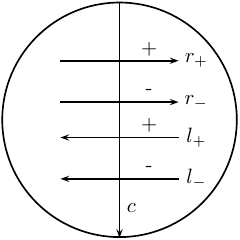}\\
\caption{The definition of the chord index}\label{figure3}
\end{figure}

Now we define the \emph{chord index} of $c$ as
\begin{center}
Ind$(c)=r_+(c)-r_-(c)-l_+(c)+l_-(c)\in\mathbb{Z}$,
\end{center}
and the \emph{writhe polynomial} of $D$ can be defined as
\begin{center}
$W_D(t)=\sum\limits_cw(c)t^{\text{Ind}(c)}-w(D)\in\mathbb{Z}[t^{\pm1}]$.
\end{center}
Here $w(c)$ denotes the sign of $c$ and $w(D)=\sum\limits_cw(c)$ denotes the writhe of $D$. It turns out that the writhe polynomial $W_D(t)$ does not depend on the choice of the knot diagram $D$, therefore it defines a virtual knot invariant \cite{CG2013}. From now on, we will use $W_K(t)$ to denote the writhe polynomial of a virtual knot $K$. Although the calculation of the writhe polynomial is quite simple, the writhe polynomial encodes plenty of useful information of virtual knots. We refer the readers to \cite{CFGMX2020} for more details.

\subsection{Another interpretation of the writhe polynomial}\label{subsection2.3}
In this subsection, we provide another interpretation of the writhe polynomial \cite[Proposition 4.5(3)]{CFGMX2020}, which motivates the definition of the multivariate writhe polynomial of multi-virtual knots given in Section \ref{section3}. The main idea is, instead of counting the contribution from other chords to the chosen chord, one can also consider the contribution coming from virtual chords.

More precisely, let $D$ be an oriented virtual knot diagram and $G(D)$ the corresponding Gauss diagram. Similar to the chords in $G(D)$ which correspond to classical crossing points, we can also add virtual chords which correspond to virtual crossing points. In order to assign an orientation to each virtual chord, we can replace each virtual crossing of $D$ with a positive classical crossing, then the orientation is induced from this positive crossing, i.e. it is directed from the preimage of the overcrossing to the preimage of the undercrossing.

For a fixed (classical) chord $c$, let us use $r_v(c)$ to denote the number of virtual chords crossing $c$ from the left side to the right side. Similarly, we use $l_v(c)$ to denote the number of virtual chords crossing $c$ from the right side to the left side. Now we define the \emph{virtual index} Ind$_v(c)$ of $c$ as 
\begin{center}
Ind$_v(c)=l_v(c)-r_v(c)$.
\end{center}

\begin{lemma}\label{lemma2.1}
The virtual index $\operatorname{Ind}_v(c)$ coincides with the chord index $\operatorname{Ind}(c)$.
\end{lemma}

\begin{proof}
After replacing each virtual crossing with a positive crossing, we obtain a classical knot diagram. It is well known that for a classical knot diagram, each crossing point has index zero. The result follows from the definition of the chord index and that of the virtual index.
\end{proof}

As a corollary, now the writhe polynomial $W_K(t)$ can be rewritten as 
\begin{center}
$W_K(t)=\sum\limits_cw(c)t^{\operatorname{Ind}_v(c)}-w(K)$.
\end{center}
As an application, if the writhe polynomial has the form $W_K(t)=\sum\limits_{i=m}^na_it^i$ $(a_m\neq0, a_n\neq0)$, then $c_v(K)\geq\max\{|m|, |n|\}$. Here $c_v(K)$ denotes the virtual crossing number of $K$, i.e. the minimal number of virtual crossing points among all the virtual knot diagrams of $K$. It means that, although we cannot see the virtual crossing points in the Gauss diagram, it is still possible to obtain some information of them from the Gauss diagram. We remark that this lower bound of virtual crossing number was first obtained by Satoh and Taniguchi in \cite{ST2014}.

\section{Multi-virtual knots and multivariate writhe polynomial}\label{section3}
\subsection{Multi-virtual knots}
Loosely speaking, a multi-virtual knot diagram contains more than one type of virtual crossings. For simplicity, in this paper we always assume that the set of types of virtual crossings is finite and use $\mathcal{T}=\{\alpha_1, \cdots, \alpha_k\}$ to denote it. Note that a multi-virtual knot diagram does not necessarily contain all types of virtual crossings. Two multi-virtual knot diagrams are equivalent if they are related by a sequence of isotopy, classical Reidemeister moves ($\Omega_1, \Omega_2, \Omega_3$ depicted in Figure \ref{figure1}) and multi-virtual detour moves. Here a multi-virtual detour move is a generalization of the detour move in virtual knot theory, which slides a strand carrying $m$ virtual crossings of the same type through a $(m, n)$-tangle from one side to the other side, resulting in $n$ virtual crossings of the same type between the strand and the tangle. Note that a loop of this strand involving a virtual crossing of the same type can be added or deleted freely. It was proved in \cite[Proposition 2.1]{KMV2026} that the multi-virtual detour move is equivalent to the multi-virtual Reidemeister moves $\Omega^v_1, \Omega^v_2, \Omega^v_3, \Omega^{cv}_3$ depicted in Figure \ref{figure4}. Now we define multi-virtual knots to be the equivalence classes of multi-virtual knot diagrams modulo classical Reidemeister moves $\Omega_1, \Omega_2, \Omega_3$ depicted in Figure \ref{figure1} and multi-virtual Reidemeister moves $\Omega^v_1, \Omega^v_2, \Omega^v_3, \Omega^{cv}_3$ depicted in Figure \ref{figure4}.

\begin{figure}
\centering
\includegraphics{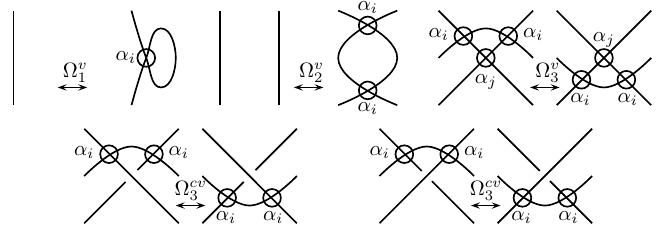}\\
\caption{Multi-virtual Reidemeister moves}\label{figure4}
\end{figure}

For a given multi-virtual knot $K$, the \emph{classical crossing number} (also called \emph{real crossing number}) $c_r(K)$ is the minimal number of classical crossings among all the diagrams representing $K$. Similarly, we define the \emph{virtual crossing number of type} $\alpha_i$ to be the smallest number of virtual crossings of type $\alpha_i$ among all the knot diagrams of $K$, which is denoted by $c_{\alpha_i}(K)$. Following \cite{KMV2026}, the \emph{total crossing number} $c_t(K)$ is defined to be the minimal number of all crossings among all the diagrams. An interesting question is, whether the equality
\begin{center}
$c_t(K)=c_r(K)+\sum\limits_{i=1}^kc_{\alpha_i}(K)$
\end{center}
holds? Obviously, a multi-virtual knot is trivial if and only if $c_t(K)=0$. Later we will find in Section \ref{section4} that, unlike the virtual knot theory, a multi-virtual knot can be nontrivial even if $c_r(K)=0$.

\subsection{Multivariate writhe polynomial}\label{subsection3.2}
In this subsection, we introduce the multivariate writhe polynomial $W_K(t_1, \cdots, t_k)\in\mathbb{Z}[t_1^{\pm1}, \cdots, t_k^{\pm1}]$ and show that it is an invariant of multi-virtual knots. Similar to the writhe polynomial $W_K(t)$, the multivariate writhe polynomial counts the signed sum of classical crossing points where each classical crossing point is equipped with an index.

More precisely, let $D$ be a knot diagram of an oriented multi-virtual knot $K$ and $G(D)$ the Gauss diagram of $D$. Similar to subsection \ref{subsection2.3}, for each virtual crossing of type $\alpha_i$ we add a virtual chord (drawn in dashed line) labeled by $\alpha_i$ to the Gauss diagram $G(D)$. The orientation of a virtual chord coincides with the orientation obtained by converting this virtual crossing into a positive classical crossing. 

Before proceeding further, we pause to point out that in virtual knot theory, the Gauss diagram completely determines the virtual knot rather than the virtual knot diagram. Actually, the generalized Reidemeister moves $\Omega_1', \Omega_2', \Omega_3'$ and $\Omega_3^s$ in Figure \ref{figure1} have no effect on the Gauss diagram. However, for multi-virtual knots, if we equip the Gauss diagram with typed virtual chords, the multi-virtual knot diagram is uniquely determined modulo sliding an arc over the point at infinity. This is because this Gauss diagram now represents a classical knot diagram, thus it completely determines the classical knot diagram on $S^2$ and hence also the multi-virtual knot diagram on $S^2$.

We now proceed to the definition of the multivariate writhe polynomial. For a (classical) chord $c$ of $G(D)$, we set 
\begin{center}
$\operatorname{Ind}_{\alpha_i}(c)=l_{\alpha_i}(c)-r_{\alpha_i}(c)$.
\end{center}
Here $l_{\alpha_i}(c)$ denotes the number of virtual chords of type $\alpha_i$ crossing $c$ from right to left and $r_{\alpha_i}(c)$ denotes the number of virtual chords of type $\alpha_i$ crossing $c$ from left to right.

\begin{definition}
Let $D$ be a multi-virtual knot diagram, we define the \emph{multivariate writhe polynomial} of $D$ to be
\begin{center}
$W_D(t_1, \cdots, t_k)=\sum\limits_{c}w(c)\prod\limits_{i=1}^kt_i^{\operatorname{Ind}_{\alpha_i}(c)}-w(D)$.
\end{center}
\end{definition}

\begin{theorem}\label{theorem3.2}
The multivariate writhe polynomial $W_D(t_1, \cdots, t_k)$ does not depend on the choice of $D$.
\end{theorem}
\begin{proof}
It suffices to verify that $W_D(t_1, \cdots, t_k)$ is invariant under the classical Reidemeister moves $\Omega_1, \Omega_2, \Omega_3$ depicted in Figure \ref{figure1} and multi-virtual Reidemeister moves $\Omega^v_1, \Omega^v_2, \Omega^v_3, \Omega^{cv}_3$ depicted in Figure \ref{figure4}. Next, we check each of these seven moves one by one.

\begin{itemize}
\item Classical Reidemeister move $\Omega_1$: the newly appeared chord $c$ is isolated, hence no virtual chord has nonempty intersection with it. It follows that $\operatorname{Ind}_{\alpha_i}(c)=0$ for any $1\leq i\leq k$ and its contribution to $W_D(t_1, \cdots, t_k)$ equals the sign of it, which is cancelled out by the writhe of the diagram.
\item Classical Reidemeister move $\Omega_2$: denote the two newly appeared chords by $c_1$ and $c_2$, then a virtual chord has nonempty intersection with $c_1$ if and only if it also has nonempty intersection with $c_2$. Since the orientations of $c_1$ and $c_2$ are almost identical, it follows that $\operatorname{Ind}_{\alpha_i}(c_1)=\operatorname{Ind}_{\alpha_i}(c_2)$. Together with the fact that $w(c_1)+w(c_2)=0$, the result follows immediately.
\item Classical Reidemeister move $\Omega_3$: according to the orientations of the three strands and the external connection configurations of $\Omega_3$, the Gauss diagram $G(D)$ admits several distinct cases. Here we only consider the case depicted in Figure \ref{figure5}. The proofs for the remaining cases follow analogously.

\begin{figure}[h]
\centering
\includegraphics{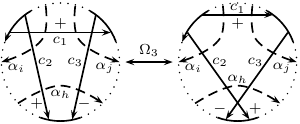}\\
\caption{Behavior of $\Omega_3$ on Gauss diagram}\label{figure5}
\end{figure}

Denote the three classical crossings involving in $\Omega_3$ by $c_1, c_2, c_3$. There are six kinds of virtual chords which may have nonempty intersection with at least one of $\{c_1, c_2, c_3\}$, say $\alpha_i, \alpha_j, \alpha_h$ and their inverses. It is easy to observe that for any virtual chord, the algebraic intersection number between it and $c_s$ $(s\in\{1, 2, 3\})$ is preserved by $\Omega_3$. Note the writhe of the knot diagram is also preserved, it follows that the multivariate writhe polynomial $W_D(t_1, \cdots, t_k)$ is unchanged under $\Omega_3$.
\item Multi-virtual Reidemeister move $\Omega^v_1$: similar to $\Omega_1$, the newly appeared virtual crossing point is isolated, thus it has no contribution to $\operatorname{Ind}_{\alpha_i}(c)$ for any classical crossing $c$.
\item Multi-virtual Reidemeister move $\Omega^v_2$: according to the orientations of the two strands, there are two possible configurations for the Gauss diagram $G(D)$ that contains two new virtual chords, see Figure \ref{figure6}.
\begin{figure}[h]
\centering
\includegraphics{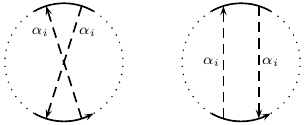}\\
\caption{Two possibilities of the Gauss diagram}\label{figure6}
\end{figure}
In each case, if a (classical) chord has nonempty intersection with one of the newly appeared virtual chord, it must has nonempty intersection with the other one. Due to the opposite orientations of these two virtual chords, for any chord $c$, their contributions to $\operatorname{Ind}_{\alpha_i}(c)$ cancel out. As a corollary, the multivariate writhe polynomial is preserved.

\item Multi-virtual Reidemeister move $\Omega^v_3$: similar to $\Omega_3$, there are several possible configurations for the Gauss diagram. Here we only consider the case depicted in Figure \ref{figure7}, other cases can be verified by a similar argument. The proof of the invariance of the multivariate writhe polynomial under $\Omega^v_3$ can essentially be regarded as the mirror version of that of $\Omega_3$. Here we have two virtual chords of type $\alpha_i$ and one virtual chord of type $\alpha_j$ involved in $\Omega^v_3$. For classical chords that have nonempty intersections with one of them, such as $c_1, c_2, c_3$ depicted in Figure \ref{figure7}, the algebraic intersection number between $\alpha_p$ $(p\in\{i, j\})$ and $c_q$ $(q\in\{1, 2, 3\})$ is invariant under $\Omega^v_3$. Hence $\operatorname{Ind}_{\alpha_p}(c_q)$ is invariant under $\Omega^v_3$, the invariance of the multivariate writhe polynomial follows.                                                                  

\begin{figure}[h]
\centering
\includegraphics{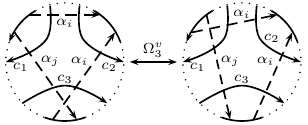}\\
\caption{Behavior of $\Omega^v_3$ on Gauss diagram}\label{figure7}
\end{figure}

\item Multi-virtual Reidemeister move $\Omega^{cv}_3$: consider the two Gauss diagrams depicted in Figure \ref{figure8}, one observes that $\operatorname{Ind}_{\alpha_i}(c)$ is preserved by $\Omega^{cv}_3$. For any other classical chord $c'$, similar as above, it is not difficult to find that $\operatorname{Ind}_{\alpha_i}(c')$ is also invariant under $\Omega^{cv}_3$. As a consequence, the multivariate writhe polynomial remains invariant under the multi-virtual Reidemeister move $\Omega^{cv}_3$.

\begin{figure}[h]
\centering
\includegraphics{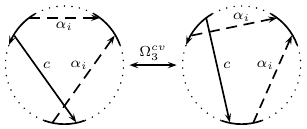}\\
\caption{Behavior of $\Omega^{cv}_3$ on Gauss diagram}\label{figure8}
\end{figure}
\end{itemize}

The proof is finished.
\end{proof}

Next we give some examples of the multivariate writhe polynomial of multi-virtual knots.

\begin{example}\label{example3.3}
Consider the four multi-virtual knots $K_1, K_2, K_3$ and $K_4$ illustrated in Figure \ref{figure9}. Direct calculation shows that
\begin{center}
$W_{K_1}=t_1t_2^{-1}-1, W_{K_2}=-t_1^{-1}t_2+1, W_{K_3}=t_1^{-1}t_2-1, W_{K_4}=-t_1t_2^{-1}+1$.
\end{center}
Thus these four multi-virtual knots are mutually distinct. Let us use $r(K)$ to denote the reverse of $K$, which is obtained from $K$ by reversing the orientation. Then we have
\begin{center}
$W_{r(K_1)}=t_1^{-1}t_2-1, W_{r(K_2)}=-t_1t_2^{-1}+1, W_{r(K_3)}=t_1t_2^{-1}-1, W_{r(K_4)}=-t_1^{-1}t_2+1$.
\end{center}
It follows that none of $K_1, K_2, K_3, K_4$ is invertible.

\begin{figure}[h]
\centering
\includegraphics{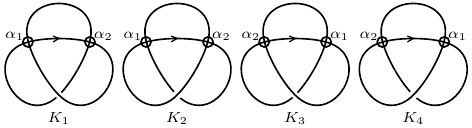}\\
\caption{Four multi-virtual knots}\label{figure9}
\end{figure}
\end{example}

\begin{example}\label{example3.4}
This example is borrowed from \cite[Subsection 6.3]{KMV2026}. Consider the closure of a 2-braid $(\sigma_1)^{2n+1}$, by replacing the first $2n$ classical crossings with virtual crossings of type $\alpha_1$ and type $\alpha_2$ alternatively, we obtain an oriented multi-virtual knot $K(n)$. An example of $K(3)$ can be found in Figure \ref{figure10}. Notice that in the Gauss diagram all virtual chords of type $\alpha_1$ cross the classical chord from the right side to the left side, and all virtual chords of type $\alpha_2$ cross the classical chord from the left side to the right side. It follows that $W_{K(n)}=t_1^nt_2^{-n}-1$.

\begin{figure}[h]
\centering
\includegraphics{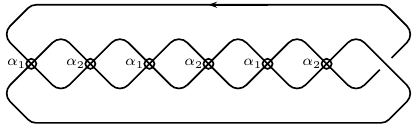}\\
\caption{A multi-virtual knot $K(3)$}\label{figure10}
\end{figure}
\end{example}

Before ending this subsection, we give an alternative interpretation of the multivariate writhe polynomial, which mimics the approach to the odd writhe polynomial given in \cite{Che2014}. Suppose we are given a multi-virtual knot diagram of $K$, say $D$, for a fixed classical crossing point $c$, we assign a $k$-tuple $(0, \cdots, 0)$ to the over-arc of it. Walking along $D$ according to the orientation, Figure \ref{figure11} tells us how to label the next arc when we meet a virtual crossing or a classical crossing. 
\begin{figure}[h]
\centering
\includegraphics{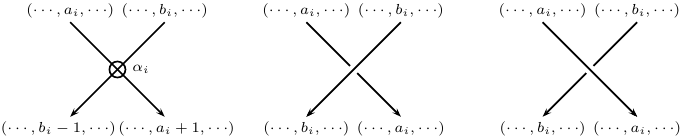}\\
\caption{Labelling rules}\label{figure11}
\end{figure}
It is not difficult to observe that this provides a well-defined labelling for each arc of $D$ via $k$-tuples. If the under-arc of $c$ receives the $k$-tuple $(a_1(c), \cdots, a_k(c))$, then we have
\begin{center}
$W_K(t_1, \cdots, t_k)=\sum\limits_cw(c)\prod\limits_{i=1}^kt_i^{a_i(c)}-w(D)$.
\end{center}
This is due to the fact that $\operatorname{Ind}_{\alpha_i}(c)=a_i(c)$, which can be easily obtained from the definition. Note that it is not necessary to assign $(0, \cdots, 0)$ to the over-arc of $c$, actually one can begin with any $k$-tuple $(a_1, \cdots, a_k)$. If the under-arc of $c$ receives $(b_1, \cdots, b_k)$, then we have $(\operatorname{Ind}_{\alpha_1}(c), \cdots, \operatorname{Ind}_{\alpha_k}(c))=(b_1-a_1, \cdots, b_k-a_k)$. Therefore, fixing a label for any one arc suffices to define the multivariate writhe polynomial.

\subsection{Basic properties of the multivariate writhe polynomial}
In this subsection, we discuss some basic properties of the multivariate writhe polynomial. 

For a multi-virtual knot diagram of $K$, say $D$, there exists a \emph{virtual projection} which maps $D$ to a virtual knot diagram $v(D)$. Here $v(D)$ is obtained from $D$ by forgetting about the type of virtual crossings. It was proved in \cite[Proposition 3.1]{KMV2026} that equivalent multi-virtual knot diagrams have equivalent virtual projections. Thus we are allowed to write $v(K)$ to denote the virtual projection of $K$. The following result shows the relation between the multivariate writhe polynomial of $K$ and the writhe polynomial of $v(K)$. We would like to remark that the writhe polynomial of $v(K)$ is nothing but the writhe polynomial corresponding to the Gauss diagram $G(D)$ without virtual chords.

\begin{proposition}\label{proposition3.5}
Let $K$ be a multi-virtual knot, then $W_{v(K)}(t)=W_K(t, \cdots, t)$.
\end{proposition}
\begin{proof}
It follows directly from Lemma \ref{lemma2.1}.
\end{proof}

For a given multi-virtual knot $K$, let us use $r(K)$ to denote the reverse of $K$ as above, and use $m(K)$ to denote the multi-virtual knot obtained from $K$ by switching all classical crossing points. Similar to the writhe polynomial \cite[Proposition 4.5]{CFGMX2020}, the multivariate writhe polynomial is quite sensitive to these two symmetries.

\begin{proposition}\label{proposition3.6}
Let $K$ be a multi-virtual knot, then we have 
\begin{center}
$W_{m(K)}(t_1, \cdots, t_k)=-W_K(t_1^{-1}, \cdots, t_k^{-1})$ and $W_{r(K)}(t_1, \cdots, t_k)=W_K(t_1^{-1}, \cdots, t_k^{-1})$.
\end{center}
\end{proposition}
\begin{proof}
Consider $m(K)$, first notice that after switching all classical crossings, the sign of every classical crossing is changed and hence $w(m(K))=-w(K)$. On the other hand, after switching each classical crossing point, the orientation of each classical chord is reversed, hence for each type $\alpha_i$, the corresponding index $\operatorname{Ind}_{\alpha_i}(c)$ changes sign. Thus we have $W_{m(K)}(t_1, \cdots, t_k)=-W_K(t_1^{-1}, \cdots, t_k^{-1})$.

For $r(K)$, the sign of each classical crossing is preserved, but the sign of $\operatorname{Ind}_{\alpha_i}(c)$ is changed. It follows immediately that $W_{r(K)}(t_1, \cdots, t_k)=W_K(t_1^{-1}, \cdots, t_k^{-1})$.
\end{proof}

Similar to the writhe polynomial, the multivariate writhe polynomial also yields lower bounds for various crossing numbers.

\begin{proposition}\label{proposition3.7}
If the multivariate writhe polynomial of a multi-virtual knot $K$ has the form 
\begin{center}
$W_K(t_1, \cdots, t_k)=\sum m_{a_1\cdots a_k}t_1^{a_1}\cdots t_k^{a_k}$,
\end{center}
where $m_{a_1\cdots a_k}\neq0$, then $c_r(K)\geq\sum\limits_{(a_1, \cdots, a_k)\neq(0, \cdots, 0)}|m_{a_1\cdots a_k}|$ and $c_{\alpha_i}(K)\geq\max\{|a_i|\}$. Here the maximum is taken over all nonzero coefficients.
\end{proposition}
\begin{proof}
This follows directly from the definition of the multivariate writhe polynomial.
\end{proof}

Let us take a revisit to Example \ref{example3.4}. Consider the following question proposed in \cite[Problem 6.6]{KMV2026}.

\begin{question}\label{question3.8}
Does the multi-virtual knot $K(n)$ has total crossing number $2n+1$?
\end{question}

Since $W_{K(n)}=t_1^nt_2^{-n}-1$, according to Proposition \ref{proposition3.7}, one observes that $c_r(K(n))\geq1$, $c_{\alpha_1}(K(n))\geq n$ and $c_{\alpha_2}(K(n))\geq n$. We conclude that $c_t(K(n))\geq 2n+1$. Since $2n+1$ total crossing points can be realized by the knot diagram $(n=3)$ depicted in Figure \ref{figure10}, we obtain $c_t(K(n))=2n+1$, which provides an affirmative answer to Question \ref{question3.8}. Similarly, Proposition \ref{proposition3.7} also tells us all the four multi-virtual knot diagrams depicted in Figure \ref{figure9} have minimal total crossing numbers.

Next let us turn to the concordance of virtual knots, which was first studied by Carter, Kamada and Saito in \cite{CKS2002}. Let $K_1\subset\Sigma_{g_1}\times[0, 1]$ and $K_2\subset\Sigma_{g_2}\times[0, 1]$ be two virtual knots embedded in thickened surfaces. We say $K_1$ and $K_2$ are \emph{concordant} if we can find a connected oriented 3-manifold $M$ such that $\partial M=-\Sigma_{g_1}\sqcup\Sigma_{g_2}$ and an annulus $A\subset M\times[0, 1]$ such that $\partial A=r(K_1)\sqcup K_2$. Or equivalently, from the viewpoint of virtual knot diagram, two virtual knots $K_1$ and $K_2$ are \emph{concordant} if they can be connected by a sequence of generalized Reidemeister moves, saddle moves, births and deaths such that the total number of births and deaths is equal to the number of saddle moves. If a virtual knot is concordant to the unknot, then it is called a \emph{slice} virtual knot. For two multi-virtual knots, the concordance can be similarly defined as the existence of a genus zero cobordism between them, see \cite[Section 6]{Kau2025}.

It was proved in \cite{BN2017} that two classical knots are virtually concordant if and only if they are concordant as classical knots. In particular, the writhe polynomial was independently proved to be a concordance invariant of virtual knots in \cite{Kau2018} and \cite{BCG2019}. The proof given by Boden, Chrisman and Gaudreau in \cite{BCG2019} is based on a key lemma proved by Turaev in \cite[Lemma 2.3.2]{Tur2008}, which claims that there exists a particular involution on the set of classical crossings if the given virtual knot is slice. In this paper, we mainly follow the approach used by Kauffman in \cite{Kau2018} to prove the following result.

\begin{proposition}\label{proposition3.9}
The multivariate writhe polynomial is a concordance invariant of multi-virtual knots.
\end{proposition}
\begin{proof}
The proof is essentially the same as that of \cite[Theorem4.9]{Kau2018}. We only give a sketch of it here.

Assume we are given two concordant multi-virtual knots, then there exists a cobordism of genus zero between them. This cobordism can be decomposed into three kinds of elementary genus zero concordances: a pair of birth with saddle, a pair of saddle with death, and a pair of two saddles, see the rightmost picture in Figure \ref{figure12}.

\begin{figure}[h]
\centering
\includegraphics{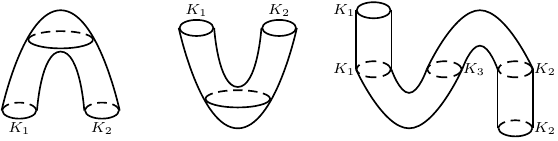}\\
\caption{Elementary genus zero concordance}\label{figure12}
\end{figure}

For the elementary concordance consisting of a birth and a saddle, it is an amalgamation of a multi-virtual knot $K$ and an unknot $U$. Recall the the multivariate writhe polynomial can be defined by arbitrary choosing the label of any one arc of $K$, see the end of subsection \ref{subsection3.2}. Let us label $(0, \cdots, 0)$ to the arc of $U$ and the arc of $K$ where the saddle point move takes place, then the new knot diagram inherits the same labeling as that of $K$. Therefore, the multivariate writhe polynomial is preserved by the pair of birth and saddle. For the pair of saddle and death, the proof is almost the same.

For the rightmost picture which contains two saddle points, in order to show that $W_{K_1}=W_{K_2}$, let us first consider the left picture in Figure \ref{figure12}. Similar to the writhe polynomial, for two multi-virtual knots, it is not difficult to observe from the definition that the multivariate writhe polynomial is additive under the connected sum. Then for the two multi-virtual knots $K_1$ and $K_2$ in the left picture in Figure \ref{figure12}, we have $W_{K_1}+W_{K_2}=W_U=0$. In a similar way, one can show that for the two multi-virtual knots $K_1$ and $K_2$ in the middle of Figure \ref{figure12}, we still have $W_{K_1}+W_{K_2}=W_U=0$. Now let us go back to the right side of Figure \ref{figure12}, one obtains $W_{K_1}=-W_{K_3}=W_{K_2}$. The proof is finished.
\end{proof}

As an example, consider the multi-virtual knot $K_1$ in Figure \ref{figure9}. Since it has nonzero multivariate writhe polynomial, we know it is not slice. On the other hand, it was shown in \cite[Figure 66]{Kau2025} that there exists a cobordism of $K_1$ to the unknot of genus one. Therefore we conclude that it has slice genus one.  

\section{Some further discussions}\label{section4}
\subsection{A map sending multi-virtual knots to multi-virtual knots}
In this subsection, we define a map sending multi-virtual knots with $k$ types of virtual crossings to multi-virtual knots with $k+1$ types of virtual crossings. One motivation to propose this map is to give an answer to \cite[Problem 3.2]{KMV2026}.

Let $K$ be a multi-virtual knot and $D$ a knot diagram of it. We define $M(D)$ to be the multi-virtual knot diagram obtained from $D$ by replacing all classical crossings with virtual crossings of type $\alpha_0$. The following lemma tells us that the map $M$ is well-defined, therefore we can use $M(K)$ to denote the image of $K$.

\begin{lemma}\label{lemma4.1}
The multi-virtual knot represented by $M(D)$ does not depend on the choice of $D$.
\end{lemma}
\begin{proof}
It suffices to notice that after replacing all classical crossings with virtual crossings of type $\alpha_0$, classical Reidemeister moves $\Omega_1, \Omega_2$ and $\Omega_3$ become $\Omega_1^v, \Omega_2^v$ and $\Omega_3^v$ $(i=j=0)$ respectively, and multi-virtual Reidemeister move $\Omega_3^{cv}$ becomes $\Omega_3^v$ $(j=0)$.
\end{proof}

Notice that $M(K)$ is a multi-virtual knot without classical crossing points. In particular, if $K$ is a classical knot, then $M(K)$ is trivial, since all crossings are virtual crossings of the same type. If any multi-virtual knot with no classical crossings is trivial, then this map has no meaningful interpretation at all. In what follows we will introduce a sequence of integer-valued invariants to demonstrate that this is not the case.

Let $D$ be a multi-virtual knot diagram, consider the Gauss diagram $G(D)$ that contains all the virtual chords. Let us use $V_i\cap V_j$ to denote the set of all the intersection points between virtual chords of type $\alpha_i$ and virtual chords of type $\alpha_j$. For any $p\in V_i\cap V_j$, we set $w(p)=+1$ if at $p$, the orientation induced by type $\alpha_i$ virtual chord and type $\alpha_j$ virtual chord coincides with the canonical orientation of the plane. Otherwise, we set $w(p)=-1$. Now we define the \emph{intersection number}
\begin{center}
$I_{ij}=\sum\limits_{p\in V_i\cap V_j}w(p)$,
\end{center}
here $1\leq i, j\leq k$. Since $I_{ji}=-I_{ij}$, we always assume $i<j$.

\begin{proposition}\label{proposition4.2}
The intersection number $I_{ij}$ $(1\leq i<j\leq k)$ is a multi-virtual knot invariant.
\end{proposition}
\begin{proof}
First we notice that $\Omega_1, \Omega_2, \Omega_3$ involve only classical crossings, thus they make no contribution to $I_{ij}$. The multi-virtual Reidemeister move $\Omega_1^v$ involves only an isolated virtual chord, which obviously preserves $I_{ij}$. For $\Omega_2^v$, see Figure \ref{figure6}, since the two involved virtual chords have opposite orientations, their respective contributions to $I_{ij}$ exactly cancel each other out. For $\Omega_3^v$, the three virtual chords involved in both sides of Figure \ref{figure7} contribute $-1$ to $I_{ij}$. If at least one of the virtual chord involved in $\Omega_3^v$ is neither of type $\alpha_i$ nor $\alpha_j$, then it is evident that $I_{ij}$ is unchanged. For the last move $\Omega_3^{cv}$, see Figure \ref{figure8}, one easily finds that $I_{ij}$ is also preserved.
\end{proof}

\begin{example}\label{example4.3}
Let us take a revisit to the multi-virtual knot $K_1$ discussed in Example \ref{example3.3}. Consider $M(K_1)$, which is a multi-virtual trefoil knot with three virtual crossings of distinct types. Direct calculations shows that $I_{01}=I_{12}=-1$ and $I_{02}=1$. Thus $M(K_1)$ is nontrivial. This provides a negative answer to \cite[Problem 3.2]{KMV2026}.
\end{example}

\begin{example}
Consider the multi-virtual knot $K(n)$ given in Example \ref{example3.4}. For $M(K(n))$, we have $I_{01}=I_{12}=-n$ and $I_{02}=n$.
\end{example}

Obviously, for any $1\leq i<j\leq k$, the intersection number $I_{ij}$ of $K$ coincides with that of $M(K)$. This explains the reason why in the two examples above we concern the intersection numbers of $M(K)$ rather than that of $K$, since $i$ can be chosen to be 0 in this case.

\subsection{The characterization of the multivariate writhe polynomial}
We first recall the following characterization of the writhe polynomial.

\begin{theorem}[\cites{ST2014,CFGMX2020}]
A polynomial $f(t)\in\mathbb{Z}[t, t^{-1}]$ can be realized as the writhe polynomial of some virtual knot if and only if $f(1)=0$ and $f'(1)=0$.
\end{theorem}

It follows immediately that a polynomial $f(t_1, \cdots, t_k)\in\mathbb{Z}[t_1^{\pm1}, \cdots, t_k^{\pm1}]$ can be realized as the multivariate writhe polynomial of some multi-virtual knots only if $f(t, \cdots, t)=0$ and $f'(t, \cdots, t)=0$ when $t=1$. An interesting question is, whether this condition is also sufficient?

\subsection{Multivariate writhe polynomial and intersection graph}
We know that the Gauss diagram completely determines the virtual knot. For a given Gauss diagram, there is a corresponding intersection graph such that each chord corresponds to a vertex and an intersection point between two chords corresponds an edge. The following theorem reveals the connection between the writhe polynomial and the intersection graph.

\begin{theorem}[\cite{Che2025}]\label{theorem4.6}
Two virtual knots have the same writhe polynomial if and only if they have equivalent intersection graphs.
\end{theorem}

In other words, the writhe polynomial precisely captures the information encoded in the intersection graph from the Gauss diagram. A natural question is, what information does the multivariate writhe polynomial extract from the Gauss diagram? Of course, here the Gauss diagram should include all the virtual chords, otherwise it would contain too little information.

Suppose we are given a multi-virtual knot diagram $D$ and the corresponding Gauss diagram $G(D)$. The \emph{intersection graph} $I(D)$ is defined as follows. First, each chord of $G(D)$ corresponds to a vertex of $I(D)$. If this chord is a virtual one of type $\alpha_i$, then the corresponding vertex is labelled with $\alpha_i$. For a classical chord, we assign the sign of it to the corresponding vertex. Second, if two chords, say $c_1$ and $c_2$, have nonempty intersection, and the orientation induced from $(c_1, c_2)$ coincides with the canonical orientation of the plane, then we add an arrow directed from the vertex corresponding to $c_1$ to the vertex corresponding to $c_2$. In particular, if a vertex corresponds to a virtual chord of type $\alpha_i$, then we say this vertex is a virtual vertex of type $\alpha_i$. Otherwise, we say it is classical. The following Figure \ref{figure13} illustrates a simple example.

\begin{figure}[h]
\centering
\includegraphics{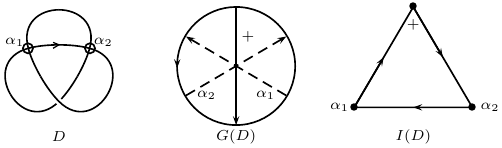}\\
\caption{A multi-virtual knot diagram, its Gauss diagram and intersection graph}\label{figure13}
\end{figure}

Choose a classical vertex $c$ of $I(D)$, it is easy to find that 
\begin{enumerate}
\item $\operatorname{Ind}_{\alpha_i}(c)$ is equal to the number of arrows pointing from virtual vertices of type $\alpha_i$ to $c$ minus the number of arrows pointing from $c$ to virtual vertices of type $\alpha_i$;
\item $I_{ij}$ is equal to the number of arrows pointing from virtual vertices of type $\alpha_i$ to virtual vertices of type $\alpha_j$ minus the number of arrows pointing from virtual vertices of type $\alpha_j$ to virtual vertices of type $\alpha_i$.
\end{enumerate}
As a consequence, the multivariate writhe polynomial $W_K(t_1, \cdots, t_k)$ and intersection numbers $I_{ij}(K)$ $(1\leq i<j\leq k)$ are all completely determined by the intersection graph. Since the intersection numbers contain some information that cannot be obtained from the multivariate writhe polynomial (see Example \ref{example4.3}), the result of Theorem \ref{theorem4.6} does not hold if we replace the writhe polynomial with multivariate writhe polynomial. We end this paper with the following problem: figure out what information of a multi-virtual knot is encoded in its intersection graph. Note that the intersection graph $I(D)$ indeed depends on the choice of the diagram $D$, in order to solve this problem, one needs to consider the equivalence classes of intersection graphs modulo the deformations derived from the classical Reidemeister moves and multi-virtual Reidemeister moves.

\section*{Acknowledgements}
Zhiyun Cheng was supported by the NSFC grant 12371065.

\end{document}